\documentclass[12pt]{article}
\usepackage[english]{babel}
\usepackage[square,numbers]{natbib}
\usepackage{amscd,float}
\usepackage{latexsym,amsmath,amssymb,amsbsy, amsthm}

\usepackage[colorlinks = TRUE, linkcolor = blue, citecolor = blue,
urlcolor = blue]{hyperref}
\usepackage{setspace}
\usepackage{lscape,fancyhdr,fancybox}
\usepackage{graphicx,epsfig, xy}
\usepackage{color}
\usepackage[hmarginratio=1:1, vmarginratio =5:5,
textheight=22cm,bindingoffset=1.5cm, textwidth=14.6cm]{geometry}
\usepackage{float}
\usepackage{graphicx}
\usepackage{caption}
\usepackage{subcaption}

\newcommand{\comment}[1]{}

\numberwithin{equation}{section}

\newtheorem{theorem}{Theorem}[section]

\newtheorem{algorithm}{Algorithm}[section]

\numberwithin{equation}{section}

\DeclareMathOperator{\Var}{Var}

\newtheorem{remark}{Remark}[section]

\theoremstyle{definition}
\newtheorem{definition}{Definition}[section]

\usepackage{amsfonts}
\newtheorem{assumption}{Assumption}[section]
\newtheorem{notation}{Notation}[section]

\usepackage{txfonts,enumerate}

\DeclareMathOperator{\E}{E} \DeclareMathOperator{\var}{Var}

\DeclareMathOperator{\Tr}{Tr}

\newcommand{\beq}{\begin{eqnarray}}
\newcommand{\eeq}{\end{eqnarray}}
\newcommand{\ben}{\begin{eqnarray*}}
\newcommand{\een}{\end{eqnarray*}}

\title{Independence of linear spectral statistics and the point process at the edge of Wigner matrices}
 \author{
\sc Debapratim Banerjee
 \\ \small Indian Statistical Institute, Kolkata\\ debopratimbanerjee@gmail.com\\}
 
\begin{document}
 \maketitle
 \begin{abstract}
 In the current paper we consider a Wigner matrix and consider an analytic function of polynomial growth on a set containing the support of the semicircular law in its interior. We prove that the linear spectral statistics corresponding to the function and the point process at the edge of the Wigner matrix are asymptotically independent when the entries of the Wigner matrix are sub-Gaussian. The main ingredient of the proof is based on a recent paper by \citet{Ban22}. The result of this paper can be viewed as a first step to find the joint distribution of eigenvalues in the bulk and the edge.
 \end{abstract}
 \section{Introduction}
 Since the groundbreaking discovery of \citet{wig1}, Wigner matrices have been a topic of key interest in the mathematics and physics communities. Later on these matrices proved to be important for many models in engineering, high dimensional statistics and many other branches.

 In particular, a Wigner matrix is a $n \times n$ symmetric (hermitian) matrix with real (complex) entries where the entries of the upper diagonal part are i.i.d. with mean $0$ and variance $\frac{1}{n}$. One is interested in the eigenvalue distribution of the matrix when the dimension grows to infinity. 
 The study of eigenvalues of Wigner matrices started with characterizing the limiting distribution of the histogram of the eigenvalues. This is done in the seminal papers of \citet{wig1} and \citet{wig2}. It is known that this limiting distribution exists and coined as the famous semicircular distribution. In particular, it is given by the following density.
\begin{equation}\label{eq:semicircular}
 f(x)=\left\{
 \begin{array}{ll}
 \frac{1}{2\pi} \sqrt{4-x^2} & \text{when~ $|x|\le 2$}\\
 0 & \text{otherwise} 
 \end{array}
 \right.
\end{equation}  
However after specifying the spectral distribution, there has been a remarkable advancement in this topic. Several interesting questions related to the eigenvalues have been raised and have been subsequently solved.

The main objective of this paper is to relate two such branches. The first one is the study of $\Tr[g(W)]$ for an analytic function $g$ containing the support of the semicircular law and the second one is the study of eigenvalues near the support of the semicircular law. Here $W$ is the Wigner matrix. The second one is about the study of the eigenvalues near the edge of the spectrum.

The study of linear spectral statistics was initiated in the paper of \citet{bai2008clt}. Here one considers a generalized Wishart matrix and the linear spectral statistics of a general analytic function. There have been several papers in this topic after that. For Wigner matrices the CLT of linear spectral statistics was studied in \citet{bai2009clt}. These results are proven using a rigorous analysis of the Stieltjes transform of the empirical distribution function of the eigenvalues of the Wigner/Wishart matrix. On the other hand a combinatorial approach for solving the similar problem was initiated in the paper by \citet{AZ05}.  

On the other hand, study of the eigenvalues near the edge of the Wigner matrix is also well known. Here we look at the point processes of the eigenvalues near $\pm 2$ which is the support of the semicircular distribution. Using the explicit distribution of the eigenvalues the fluctuation of the largest eigenvalue of the Wigner matrix was first proved in the seminal works of Tracy and Widom \cite{tracy1994level}, \cite{tracy1996orthogonal} in the Gaussian case. In particular it is proved that 
\begin{equation}
\mathbb{P}\left[n^{\frac{2}{3}} \left( \lambda_{1,n}- 2 \right)\le s\right] \to F_{\beta}(s)
\end{equation}
where the Tracy-Widom distribution functions $F_{\beta}$ can be described by Painleve equations, and $\beta= 1,2,4$ corresponds to Orthogonal/Unitary/Symplectic ensemble, respectively. Here $\lambda_{1,n}\ge \ldots \ge \lambda_{n,n}$ are the eigenvalues of the Wigner matrix $W$. The joint distribution
of $k$ largest eigenvalues can be expressed in terms of the Airy kernel, which was shown by \citet{forrester1993spectrum}. In general the joint distribution of $\left( \lambda_{1,n},\ldots , \lambda_{k,n} \right)$ will also converge after proper rescaling and centering. In particular 
\begin{equation}
\begin{split}
\mathbb{P}\left[ n^{\frac{2}{3}}\left( \lambda_{1,n}-2 \right)\le s_{1},\ldots, n^{\frac{2}{
3}}\left( \lambda_{k,n} -2\right)\le s_{k} \right] \to F_{\beta,k}(s_{1},\ldots , s_{k}). 
\end{split}
\end{equation}
The $k$ dimensional distribution $F_{\beta,k}$ will also be coined as Tracy Widom distribution.
  
After the Gaussian case being solved, a significant amount of work has been done for general non-Gaussian entries. This problem is known as edge universality. The edge universality of the Wigner matrix was first proved in the paper \citet{sosh} through combinatorial methods. He assumed that the distributions of the entries of the matrix are sub-Gaussian and symmetric. The universality for the non-symmetric entries was first proved by \citet{tao2010random}. Here one assumes that entries have vanishing third moment and the tail decays exponentially. Then there was a different approach by Erd\H{o}s, Yau and others by analyzing the resolvent matrix. One might look at \citet{erdHos2017dynamical}. Very recently \citet{Ban22} proved edge universality through a new combinatorial approach where he removed assumption of symmetry of the entries.

However to the best of our limited knowledge nothing has been proved for the joint distribution of a generic linear spectral statistics and the eigenvalues near the edge. \emph{In this paper following the combinatorial approach of \citet{Ban22} and \citet{AZ05} we prove that they are asymptotically independent.} This result can be seen as a first step to prove the joint distribution of eigenvalues near a typical point in the interior of the support of the semicircular law and the eigenvalues near the edge. This is due to the following reason: we choose a point $x \in (-2,2)$ and look at an interval of length $O\left(\frac{1}{n}\right)$ around $x$ and look at the eigenvalues falling in the interval. This type of statistics can be approximated by a function $g_{x,n}(\cdot)$ where $g_{x,n}(\cdot)$ takes $O(1)$ values in an interval of length $O\left(\frac{1}{n}\right)$ around $x$ and is very close to $0$ outside the interval. However as of now, there are some technical difficulties in this scenario. This is due to the fact that the main ingredient of the proof is by approximation of an analytic function by a first few terms of it's power series expansion. However for analytic functions just discussed, this kind of approximation is not well behaved. This is an ongoing work. 

The linear spectral statistics and the largest eigenvalue are also interesting from high dimensional statistics and spin glass point of view. In spin glasses, it is known that for spherical Sherrington Kirkpatrick model the free energy at high temperature can be approximated by a linear spectral statistics of the interaction matrix and in low temperature it can be approximated by the largest eigenvalue of the interaction matrix. One might look at \citet{Baikleestat} for a reference. A straight forward corollary of our result is the independence of free energies in low and high temperature of spherical SK model made out of same interaction matrix. The joint distribution of largest eigenvalue and a generic linear spectral statistics is also important in high dimensional statistics. In particular it is proved in \citet{zhang2020asymptotic} for large spikes, when we consider the largest eigenvalue of a spiked Wishart matrix and a generic spectral statistics they are independent. One might also look at \citet{li2020asymptotic}. However these results are for the spiked matrices where the spikes are large. In particular, the largest eigenvalue has an asymptotic normal distribution as opposed to Tracy-Widom distribution.  
 \section{Preliminaries}
 \subsection{Model description and related terminologies}
 Firstly we start with the definition of Wigner matrices. 
 \begin{definition}\label{def:wig}
 We call a matrix $W= \left( x_{i,j}/ \sqrt{n} \right)_{1\le i,j \le n}$ to be a Wigner matrix if $x_{i,j}=\bar{x}_{j,i}$, $\left(x_{i,j}\right)_{1\le i<j\le n}$ are i.i.d., $\E[x_{i,j}]=0$ and $\E[|x_{i,j}|^2]=1$. 
 \end{definition}
 In this paper we deal with the real symmetric matrices and for the ease of calculation we scale the whole matrix by a factor $2$. With slight abuse of notation we shall also call this matrix a Wigner matrix and denote it by $W$. Following are the assumptions of the matrices we consider.
 \begin{assumption}\label{ass:wig}
 We consider a matrix $W$ given by $W= (x_{i,j}/\sqrt{n})_{1\le i,j \le n}$ such that the following conditions are satisfied:
 \begin{enumerate}[(i)]
 \item $x_{i,j}=x_{j,i}$ for $i\le j$.
 \item $\var(x_{i,j})=\frac{1}{4}$
 \item $\left(x_{i,j}\right)_{1\le i<j \le n}$ are i.i.d.
 \item $\E[x_{i,j}^{2k}]\le \left( \mathrm{const} .k \right)^{k}$ $\forall ~ k \in \mathbb{N}$ 
 \end{enumerate}
 \end{assumption}
 Given a Wigner matrix $W$ of dimension $n \times n$ we denote its eigenvalues by $\lambda_{1,n}\ge \ldots \ge \lambda_{n,n}$. It is well known that for a Wigner matrix in Definition \ref{def:wig}, the measure $\frac{1}{n}\sum_{i=1}^{n} \delta_{\lambda_{i,n}}$ converge weakly to the semicircular law in the almost sure sense. The law is given by the following density:
\begin{equation}\label{eq:semicircular}
 f(x)=\left\{
 \begin{array}{ll}
 \frac{1}{2\pi} \sqrt{4-x^2} & \text{when~ $|x|\le 2$}\\
 0 & \text{otherwise.} 
 \end{array}
 \right.
\end{equation} 
  When we scale the entries by a factor $2$, the distribution is supported in $[-1,1]$ and its density is given by 
 \begin{equation}\label{eq:semicircularII}
 f(x)=\left\{
 \begin{array}{ll}
 \frac{2}{\pi} \sqrt{1-x^2} & \text{when~ $|x|\le 1$}\\
 0 & \text{otherwise.} 
 \end{array}
 \right.
\end{equation}
One of the main quantities of consideration in this paper is the linear spectral statistics. This is defined as follows:
\begin{definition}
Suppose $g$ is an analytic function on a set containing $[-1,1]$ in its interior. We define the linear spectral statistics corresponding to the function $g$ as follows:
\begin{equation}
\Tr[g(W)]= \sum_{i=1}^{n} g(\lambda_{i,n})
\end{equation}
\end{definition}
Our main approach is to work with a polynomial function and use the power series representation of the polynomial function to work with general analytic function. So most of the time $g$ will be a polynomial of finite degree.
\subsection{Preliminaries for the combinatorial approach}
The main approach of this paper is combinatorial. For the linear spectral statistics we follow the approach of \citet{AZ05} and for the point process near the edge we follow the approach of \citet{Ban22}. In this subsection we develop the required terminologies.
To begin with we start with a matrix $W$ of dimension $n \times n$. Its $k$ th moment is given by 
\begin{equation}\label{eq:traceexpression}
\Tr[W^{k}]= \sum_{i_{0},i_{1},\ldots ,i_{k-1}, i_{0}} W_{i_{0},i_{1}}\ldots W_{i_{k-1},i_{0}}.
\end{equation} 
The word sentence method systematically analyzes the tuples $(i_{0}, \ldots, i_{k-1},i_{0})$ for some suitable $k$. To do this we need some notations and definitions. 
In this part we give a very brief introduction to words, sentences and their equivalence classes essential for the combinatorial analysis of random matrices. 
The definitions are taken from \citet{AGZ} and \citet{AZ05}.
For more general information, see \cite[Chapter 1]{AGZ} and \cite{AZ05}. 
\begin{definition}[$\mathcal{S}$ words]
Given a set $\mathcal{S}$, an $\mathcal{S}$ letter $s$ is simply an element of
$\mathcal{S}$. An $\mathcal{S}$ word $w$ is a finite sequence of letters $s_1
\ldots s_k$, at least one letter long.
An $\mathcal{S}$ word $w$ is \emph{closed} if its first and last letters are the same. In this paper $\mathcal{S}=\{1,\ldots,n\}$ where $n$ is the number of nodes in the graph.
\end{definition}
Two $\mathcal{S}$ words
$w_1,w_2$ are called \emph{equivalent}, denoted $w_1\sim w_2$, if there is a bijection on $\mathcal{S}$ that
maps one into the other.
For any word $w = s_1 \ldots s_k$, we use $l(w) = k$ to denote the \emph{length} of $w$, define
the \emph{weight} $wt(w)$ as the number of distinct elements of the set ${s_1,\ldots , s_k
}$ and the
\emph{support} of $w$, denoted by $\mathrm{supp}(w)$, as the set of letters appearing in $w$. 
With any word $w$ we may associate an undirected graph, with $wt(w)$ vertices and at most $l(w)-1$ edges,
as follows.
\begin{definition}[Graph associated with a word] 
	\label{def:graphword}
Given a word $w = s_1 \ldots s_k$,
we let $G_w = (V_w,E_w)$ be the graph with set of vertices $V_w = \mathrm{supp}(w)$ and (undirected)
edges $E_w = \{\{s_i, s_{i+1}
\}, i = 1,\ldots ,k - 1
\}.$
\end{definition}
The graph $G_w$ is connected since the word $w$ defines a path connecting all the
vertices of $G_w$, which further starts and terminates at the same vertex if the word
is closed. 
We note that equivalent words generate the same
graphs $G_w$ (up to graph isomorphism) and the same passage-counts of the edges. 
Given an equivalence class $\mathbf{w}$, we shall sometimes denote $\#E_{\mathbf{w}}$ and $\#V_{\mathbf{w}}$ to be the common number of edges and vertices for graphs associated with all the words in this equivalence class $\mathbf{w}$. 

\begin{definition}[Weak Wigner words]
	\label{def:weakwigner}
Any word $w$ will be called a \emph{weak Wigner word} if the following conditions are satisfied:
\begin{enumerate}
\item $w$ is closed.
\item $w$ visits every edge in $G_{w}$ at least twice. 
\end{enumerate}
\end{definition}
Suppose now that $w$ is a weak Wigner word. If $wt(w) = (
l(w) + 1)/2$, then we
drop the modifier ``weak" and call $w$ a \emph{Wigner word}. (Every single letter word is
automatically a Wigner word.) 
Except for single letter words, each edge in a Wigner word is traversed exactly twice.
If $wt(w) = (l(w)-1)/2$, then we call $w$ a \emph{critical
weak Wigner word}.

It is a well known result in random matrix theory that there is a bijection from the set of the Wigner words of length $2k+1$ to the set of Dyck paths of length $2k$.
We now move to definitions related to sentences.

\begin{definition}[Sentences and corresponding graphs]
	\label{def:sentence}
A sentence $a=[w_i]_{i=1}^{m}=[[\alpha_{i,j}]_{j=1}^{l(w_i)}]_{i=1}^{m}$ is an ordered collection of $m$ words of length $(l(w_1),\ldots,l(w_m))$ respectively. We define the graph $G_a=(V_a,E_a)$ to be the graph with 
\[
V_a= \mathrm{supp}(a),\quad 
E_a= \left\{ \{ \alpha_{i,j},\alpha_{i,j+1}\}| i=1,\ldots,m ; j=1,\ldots, l(w_i)-1 \}  \right\}. 
\] 
\end{definition} 
\begin{notation}
Given a closed word $w=(i_{0},i_{1},\ldots,i_{k-1},i_{k}=i_{0})$ we need to consider the random variable $\prod_{j=0}^{k-1} x_{i_{j},i_{j+1}}$. We call this random variable $X_{w}$.
\end{notation} 
In addition to this we also need the map from word to Dyck paths and the related results and terminologies for the analysis of the eigenvalues near the edge. For these the reader may have a look at subsection 7.2 in \citet{Ban22}.
\section{Main result}
We now state the main result of this paper.
\begin{theorem}\label{thm:main}
Suppose we have a Wigner matrix $W$ satisfying the assumptions in Assumption \ref{ass:wig}. Then the following are true:
\begin{enumerate}
\item Let $g$ be any polynomial of fixed degree. Then the random variable\\
 $\left(\Tr[g(W)]-\E\left[\Tr[g(w)]\right]\right)$ and the point process at the edge of the spectrum (i.e. at $x= \pm 1$) are asymptotically independent. 
\item More generally $g$ can be taken to be a fixed analytic function of polynomial growth on a set containing $[-1,1]$ in its interior.
\item As a corollary, for any fixed $k$, the random vector $\left(n^{\frac{2}{3}}2\left(\lambda_{1,n}-1\right), \ldots, n^{\frac{2}{3}}2\left(\lambda_{k,n} -1 \right) \right)$\\
 and $\left(\Tr[g(W)]-\E\left[\Tr[g(w)]\right]\right)$ are independent.
\end{enumerate}
\end{theorem}
\section{Strategies for the proof}
As mentioned at the beginning, the main technique of this current paper is combinatorial in nature. In particular for an analytic function $g$, it is known that\\
 $\left(\Tr\left[ g(W) \right]-\E\left[ \Tr\left[ g(W) \right] \right]\right)$ converges to a centered Gaussian with explicit variance. The combinatorial approach for proving this kind of result was initiated in the paper by \citet{AZ05}. In order the prove this the authors first consider the joint distribution of 
$$\left( \Tr\left[ W^{m_{1}} \right] - \E\left[\Tr\left[ W^{m_{1}} \right]\right],\ldots, \Tr\left[ W^{m_{k}} \right] - \E\left[\Tr\left[ W^{m_{k}} \right]\right] \right)  $$ 
 for every fixed $k$ and integers $m_{1},\ldots, m_{k} \in \mathbb{N}$. It is proved using the method of moments that this joint distribution is $k$ dimensional Gaussian with explicit covariance matrix. The results are summarized in the following theorem. 
 \begin{theorem}\label{thm:polyclt}(Chapter $1$ \citet{AGZ})
 Suppose $W$ is a Wigner matrix satisfying Assumption \ref{ass:wig} and we want to determine the joint distribution of
 $$\left( \Tr\left[ W^{m_{1}} \right] - \E\left[\Tr\left[ W^{m_{1}} \right]\right],\ldots, \Tr\left[ W^{m_{k}} \right] - \E\left[\Tr\left[ W^{m_{k}} \right]\right] \right).$$
   Then there exists $\Sigma$ such that 
 \begin{enumerate}
 \item for any $l$ and $m_{1}', \ldots , m_{l}' \in \{ m_{1},\ldots, m_{k} \}$ the joint moment 
 \begin{equation}
 \begin{split}
 &\E\left[\prod_{i=1}^{l} \left(\Tr\left[ W^{m_{i}'} \right] - \E\left[\Tr\left[ W^{m_{i}'} \right]\right]\right) \right]\\
 & \to \left\{
 \begin{array}{ll}
 0 & \text{if $l$ is odd}\\
 \sum_{\eta\in \mathcal{P}_{l}}\prod_{j=1}^{\frac{l}{2}} \Sigma\left( m_{\eta(j,1)}',m_{\eta(j,2)}' \right) & \text{otherwise}.
 \end{array} 
 \right.
 \end{split}
 \end{equation}
 Here for an even number $l$, $\mathcal{P}_{l}$ denotes all partitions of the set $\{ 1,2,\ldots, l \}$ such that each block has exactly two elements. Further for any partition $\eta \in \mathcal{P}_{l}$, $\eta(j,1)$ and $\eta(j,2)$ denote the first and second element of the $j$ th block in $\eta$. 
 \item As a corollary, we have the vector $$\left( \Tr\left[ W^{m_{1}} \right] - \E\left[\Tr\left[ W^{m_{1}} \right]\right],\ldots, \Tr\left[ W^{m_{k}} \right] - \E\left[\Tr\left[ W^{m_{k}} \right]\right] \right)  $$ converges to a $k$ dimensional centered Gaussian with covariance matrix $\left(\Sigma(m_{i},m_{j})\right)_{1\le i,j \le k}$.
\end{enumerate}  
 \end{theorem}
For an explicit expression of $\Sigma$ one might look at section 2.1.7 of \citet{AGZ}.

Finally for a general analytic function $g(x)= \sum_{i=0}^{\infty} g_{i}x^{i}$, one considers the following polynomial $g^{(m)}(x)= \sum_{i=0}^{m} g_{i}x^{i}$ and prove the following result:
\begin{theorem}(\citet{AZ05})\label{thm:cltapprox}
Suppose $W$ be a Wigner matrix satisfying Assumption \ref{ass:wig}. Then for any analytic function $g(x)=\sum_{i=0}^{\infty}g_{i}x^{i}$ on a set containing $[-1-\varepsilon,1+\varepsilon]$ and having polynomial growth we have
\begin{equation}
\Var\left[\Tr[g(W)]  \right]\le \bar{c} \sup_{x\in [-1-\varepsilon,1+\varepsilon]} |g'(x)|^2.
\end{equation}
In particular if the function is analytic on a interval containing $[-1-\varepsilon, 1+\varepsilon]$, for any $\eta$ we can choose large enough $i$ such that 
\begin{equation}
\Var\left[ \Tr[g(W)]- \sum_{j=0}^{i} g_{j}x^{j}  \right]\le \eta.
\end{equation}
\end{theorem}
\begin{remark}
One might note that for Theorem \ref{thm:cltapprox} to hold, \citet{AZ05} require the entries to satisfy Poincare inequality. However one might look at Remark 3.12 of \citet{van2014probability} to see that sub-Gaussianity implies Poincare inequality.
\end{remark}
On the other hand in the recent paper \citet{Ban22} the combinatorial approach for proving edge universality with general non-symmetrically distributed entries was considered. This method relies upon proving the following results
\begin{theorem}(\citet{Ban22})\label{thm:traceconvergence}
Consider the Wigner matrix $W$ satisfying Assumption \ref{ass:wig}. Then for any fixed $t \in (0,\infty)$ taking $k= \left[tn^{\frac{2}{3}}\right]$, we have the following results 
\begin{enumerate}
\item $\E\Tr\left[ W^{2k} \right]=O(1)$ and $\E\Tr\left[W^{2k+1}\right]=o(1)$.
\item If the limit of $\lim_{n \to \infty} \E\Tr\left[ W^{2k} \right] $ for some $t\in (0,\infty)$ exists, then the limit only depends on the first and second moment of entries. 
\item As the limit exists for Gaussian entries, the limit exists and is universal for any Wigner matrix satisfying Assumption \ref{ass:wig}. 
\end{enumerate}
\end{theorem}
\begin{theorem}(\citet{Ban22})\label{thm:jointmom}
Consider the Wigner matrix $W$ satisfying Assumption \ref{ass:wig}. Then for any fixed $t_{1}, \ldots, t_{l} \in (0,\infty)^{l} $ taking $k_{i}= \left[ t_{i} n^{\frac{2}{3}} \right]$, we have the following results
\begin{enumerate}
\item 
\begin{equation}\label{eq:jointlim}
\E\left[ \prod_{i=1}^{l} \left[\Tr\left[W^{k_{i}}\right] - \E\left[  \Tr\left[W^{k_{i}}\right]\right] \right] \right]=O(1).
\end{equation}
\item If the limit in \eqref{eq:jointlim} exists for some $t_{1},\ldots , t_{l}$, then the limit only depends on the first and second moment of entries.
\item As the limit exists for Gaussian entries, the limit exists and is universal for any Wigner matrix satisfying Assumption \ref{ass:wig}.
\end{enumerate}
\end{theorem}
Theorems \ref{thm:traceconvergence} and \ref{thm:jointmom} is enough to prove the edge universality as the laplace transformation of the $l$ point correlation function of the point process at the edge of the spectrum is determined by the limits in Theorems \ref{thm:jointmom} and \ref{thm:traceconvergence}. As weak convergence of correlation functions of a sequence of point processes determines the weak convergence of the sequence of point processes, we get the desired result of edge universality.

So in order to prove the independence of the linear spectral statistics and the point process at the edge, we prove the following result:
\begin{theorem}\label{thm:joincltpp}
Suppose we fix $m_{1},\ldots, m_{k}$. Then for any $l$, $l'$, $m_{1}',\ldots, m_{l}' \in \{m_{1},\ldots, m_{k} \}$ and $t_{1},\ldots, t_{l'} \in (0,\infty)$, the following holds:
\begin{equation}
\begin{split}
&\E\left[\prod_{i=1}^{l} \left(\Tr\left[ W^{m_{i}'} \right] - \E\left[\Tr\left[ W^{m_{i}'} \right]\right]\right) \prod_{j=1}^{l'}\left(\Tr\left[W^{[t_{j}n^{\frac{2}{3}}]}\right] - \E\left[  \Tr\left[W^{[t_{j}n^{\frac{2}{3}}]}\right]\right] \right)  \right]\\
& \to \lim_{n \to \infty} \E\left[ \prod_{i=1}^{l} \left(\Tr\left[ W^{m_{i}'} \right] - \E\left[\Tr\left[ W^{m_{i}'} \right]\right]\right) \right]\lim_{n \to \infty}\E\left[\prod_{j=1}^{l'}\left(\Tr\left[W^{[t_{j}n^{\frac{2}{3}}]}\right] - \E\left[  \Tr\left[W^{[t_{j}n^{\frac{2}{3}}]}\right]\right] \right)   \right]
\end{split}
\end{equation}
\end{theorem}
Once this is proved, Theorem \ref{thm:main} readily follows with the help of Theorem \ref{thm:cltapprox}.
\section{Proofs of the results}
To begin with we state an algorithm. This algorithm takes two closed words $w_{1}$ and $w_{2}$ of lengths $k_{1}+1$ and $k_{2}+1$ as input such that the words have at least one edge in common and gives a closed word $w_{3}$ of length $k_{1}+k_{2}+1$ as an output which has the same edge set as the union of the edges of $w_{1}$ and $w_{2}$. This is taken from \citet{Ban22}
\begin{algorithm}\label{alg:embed}
We start with two words $w_{1}$ and $w_{2}$ such that $w_{1}$ and $w_{2}$ shares an edge. Let $\{ \alpha, \beta \}$ be the first edge in $w_{2}$ which is repeated in $w_{1}$. We consider the first appearance of $\{\alpha, \beta \}$ in $w_{2}$. Without loss of generality we assume that the  first appearance of the edge $\{ \alpha , \beta \}$ appears in the word $w_{2}$ in the order $(\alpha,\beta)$. We now consider any appearance (for concreteness say the first) of the edge $\{ \alpha, \beta \}$ in the word $w_{1}$. This appearance $\{ \alpha ,\beta \}$ can be traversed in $w_{1}$ in the order $(\alpha,\beta)$ or $(\beta, \alpha)$.
Considering these we have the word $w_{2}$ looks like 
\begin{equation}\label{eq:w1before}
w_{2}= (\alpha_{0},\alpha_{1},\ldots,\alpha_{p_{1}}, \alpha, \beta,\ldots,\alpha_{k_{1}-1},\alpha_{0})
\end{equation} 
and the word $w_{1}$ looks like 
\begin{equation}\label{eq:w21}
w_{1}= \left( \beta_{0}, \beta_{1},\ldots,\beta_{q_{1}},\alpha, \beta, \ldots, \beta_{k_{2}-1},\beta_{0} \right)
\end{equation}
or
\begin{equation}\label{eq:w22} 
 w_{1}=(\beta_{0},\beta_{1},\ldots,\beta_{q_{1}},\beta,\alpha,\ldots,\beta_{k_{2}-1},\beta_{0}).
 \end{equation}
Now we output the word $w_{3}$ as follows:
\begin{enumerate}
\item Suppose $w_{1}$ is of the form \eqref{eq:w21}, then 
\begin{equation}\label{eq:w31}
w_{3}= \left(  \alpha_{0},\alpha_{1},\ldots,\alpha_{p_{1}},\alpha,\beta,\beta_{q_{1}+3},\ldots, \beta_{k_{2}-1}, \beta_{0},\ldots,\beta_{q_{1}},\alpha,\beta,\alpha_{p_{1}+3},\ldots,\alpha_{k_{1}-1},\alpha_{0} \right).
\end{equation}
\item On the other hand when $w_{1}$ is of the form \eqref{eq:w22}, 
\begin{equation}\label{eq:w32}
w_{3}= \left( \alpha_{0},\alpha_{1},\ldots,\alpha_{p_{1}},\alpha,\beta,\beta_{q_{1}},\ldots,\beta_{0},\beta_{k_{2}-1},\ldots,\beta_{q_{1}+3},\alpha,\beta,\alpha_{p_{1}+3},\ldots,\alpha_{k_{1}-1},\alpha_{0}\right). 
\end{equation}
\end{enumerate}
\end{algorithm}
\begin{proof}[Proof of Theorem \ref{thm:joincltpp}]
As this is the most important proof in this paper, we prove it in step by step. \\
\textbf{Step 1 (A basic observation):} We begin with a very basic observation.
\begin{equation}
\begin{split}
&\E\left[\prod_{i=1}^{l} \left(\Tr\left[ W^{m_{i}'} \right] - \E\left[\Tr\left[ W^{m_{i}'} \right]\right]\right) \prod_{j=1}^{l'}\left(\Tr\left[W^{[t_{j}n^{\frac{2}{3}}]}\right] - \E\left[  \Tr\left[W^{[t_{j}n^{\frac{2}{3}}]}\right]\right] \right)  \right]\\
&= \left( \frac{1}{n} \right)^{\sum_{i}\frac{m_{i}'}{2}+\sum_{j}\frac{[t_{j}n^{\frac{2}{3}}]}{2}}\sum_{w_{1},\ldots, w_{l}~|~ l(w_{i})=m_{i}'+1}\sum_{w_{1}',\ldots, w_{l'}'~|~ l(w_{j}')=[t_{j}n^{\frac{2}{3}}]+1}\E\left[\prod_{i=1}^{l}\left(X_{w_{i}}-E[X_{w_{i}}]\right)\prod_{j=1}^{l'}\left( X_{w'_{j}}-\E[X_{w'_{j}}] \right)\right]\\
&=\left( \frac{1}{n} \right)^{\sum_{i}\frac{m_{i}'}{2}+\sum_{j}\frac{[t_{j}n^{\frac{2}{3}}]}{2}}\sum_{a_{1}}\sum_{a_{2}}\E\left[\prod_{i=1}^{l}\left(X_{w_{i}}-E[X_{w_{i}}]\right)\prod_{j=1}^{l'}\left( X_{w'_{j}}-\E[X_{w'_{j}}] \right)\right]\\
&= \underbrace{\left( \frac{1}{n} \right)^{\sum_{i}\frac{m_{i}'}{2}+\sum_{j}\frac{[t_{j}n^{\frac{2}{3}}]}{2}} \left( \sum_{a_{1},a_{2}~|~E_{a_{1}}\cap E_{a_{2}\neq \emptyset}}\E\left[\prod_{i=1}^{l}\left(X_{w_{i}}-E[X_{w_{i}}]\right)\prod_{j=1}^{l'}\left( X_{w'_{j}}-\E[X_{w'_{j}}] \right)\right] \right)}_{T_{1}(say)}+\\
&~~~~~~~~~~~~~~~ \underbrace{\left( \frac{1}{n} \right)^{\sum_{i}\frac{m_{i}'}{2}+\sum_{j}\frac{[t_{j}n^{\frac{2}{3}}]}{2}} \left( \sum_{a_{1},a_{2}~|~E_{a_{1}}\cap E_{a_{2}=\emptyset}}\E\left[\prod_{i=1}^{l}\left(X_{w_{i}}-E[X_{w_{i}}]\right)\prod_{j=1}^{l'}\left( X_{w'_{j}}-\E[X_{w'_{j}}] \right)\right] \right)}_{T_{2}(say)}
\end{split}
\end{equation}
Here $a_{1}=[w_{i}]_{i=1}^{l}$ and $a_{2}=[w_{j}']_{j=1}^{l'}$. We shall analyze $T_{1}$ and $T_{2}$ separately.\\
\textbf{Step 2 (Analysis of $T_{1}$)}
$T_{1}$ contains the cases when $E_{a_{1}} \cap E_{a_{2}} \neq \emptyset$.

\noindent 
 For this without loss of generality we assume $a_{1}$ contains a single word and so does $a_{2}$. The general case can be proved by repeated use of the arguments given next. We assume $w_{1}$ be a word of fixed length $m+1$ while $w_{2}$ is a word of length $[tn^{\frac{2}{3}}]+1$ for some $t \in (0,\infty)$ such that $E_{w_{1}} \cap E_{w_{2}} \neq \emptyset$. We would like to find the following expectation 
\begin{equation}\label{eq:jointshare}
\begin{split}
\left( \frac{1}{n} \right)^{\frac{m+[tn^{\frac{2}{3}}]}{2}} \sum_{w_{1},w_{2}~|~ l(w_{1})= m+1, l(w_{2})= [tn^{\frac{2}{3}}]+1, E_{w_{1}}\cap E_{w_{2}}\ne \emptyset}\left(\E[X_{w_{1}}X_{w_{2}}]-\E[X_{w_{1}}]\E[X_{w_{2}}]\right)
\end{split}
\end{equation}

We now apply Algorithm \ref{alg:embed} to $(w_{1},w_{2})$ to get a word $\mathfrak{w}$. Hence we write the sum in \eqref{eq:jointshare} in the following way:
\begin{equation}
\left( \frac{1}{n} \right)^{\frac{m+[tn^{\frac{2}{3}}]}{2}}\sum_{\mathfrak{w}}\sum_{w_{1},w_{2}~|~ f(w_{1},w_{2})=\mathfrak{w}} \left(\E[X_{w_{1}}X_{w_{2}}]-\E[X_{w_{1}}]\E[X_{w_{2}}]\right).
\end{equation}
In order to have a nontrivial value of $\left(\E[X_{w_{1}}X_{w_{2}}]-\E[X_{w_{1}}]\E[X_{w_{2}}]\right)$, one needs to have each edge in $\mathfrak{w}$ repeated at least twice. We call this set $\mathcal{W}_{\ge 2, m+[tn^{\frac{2}{3}}]}$. By a straight forward application of Holder's inequality we know $\E[|X_{w_{1}}X_{w_{2}}|]\ge \E[|X_{w_{1}}|]\E[|X_{w_{2}}|].$ So it is enough to bound the following quantity:
\begin{equation}\label{jointtobound}
\left( \frac{1}{n} \right)^{\frac{m+[tn^{\frac{2}{3}}]}{2}}\sum_{\mathfrak{w}\in \mathcal{W}_{\ge 2, m+[tn^{\frac{2}{3}}]}}2|f^{-1}(\mathfrak{w})|\E[|X_{\mathfrak{w}}|].
\end{equation}
 We now look closely at the proof of Theorem 6.2 in \citet{Ban22}. Imitating the arguments of the proof of this theorem it can be showed that all the other cases apart from case (i) have negligible contribution. Also following the proofs of Propositions 8.2 and 8.3 in \citet{Ban22}, we can restrict our attention to the case when every edge in $\mathfrak{w}$ is traversed exactly twice and every type $j \ge 2$ instant in the skeleton word is actually a type $2$ instant. Now we analyze this case in details.
 
\noindent
Alike \citet{Ban22} we at first fix the skeleton word. Let $N$ be the number of type $j\ge 2$ instants in the skeleton word. Given a skeleton word, we choose a type $j\ge 2$ instant and find out the edge which is closed in the next instant. We choose the instant which corresponds to the left endpoint of the first traversal of this edge. So given the skeleton word there is at most $N$ choices to choose the endpoints of $w_{1}$. Let this choice be $l$. Let the length of the part of the skeleton word inside these two points be $m_{1}$ and the outside these two points be $m_{2}$. As a consequence, in $w_{2}$ there will be exactly $m_{2}+1$ Dyck paths and in $w_{1}$ there will be exactly $m_{1}$ Dyck paths. So this might appear same as the calculation in Theorem 6.2 in \citet{Ban22}. However as the length of the word $w_{1}$ is fixed, there will be additional constraints. In particular, if the first endpoint of the word  is the instant where this level first appears in the skeleton word, we have $|p_{1}+\ldots + p_{l}-r_{l}|\le m$. Here we chose the $l$ th type $j \ge 2$ instant. On the other hand, it is possible that the skeleton word returns to the level of first chosen point multiple times before reaching the second chosen point. In this case we have $p_{l}\le m$.    As all the type $j\ge 2$ instants in the skeleton word are type $2$, we have the following upper bound to \eqref{jointtobound}:
\begin{equation}\label{eq:tobound}
\begin{split}
&C\left( \frac{1}{2} \right)^{m+[tn^{\frac{2}{3}}]}\sum_{N}\sum_{l=1}^{N}\sum_{p_{1},\ldots,p_{N}}\sum_{q_{1},\ldots,q_{N-1}}\sum_{r_{1},\ldots,r_{N}~:~|p_{1}+\ldots +p_{l}-r_{l}|\le m} n^{\frac{m+[tn^{\frac{2}{3}}]}{2}-N+1}3^{N}\\
&~~~~~~~~~\left(  m+1\right)\frac{m_{1}}{m+1} \binom{m+1}{\frac{m_{1}+m+1}{2}}\frac{m_{2}+1}{[tn^{\frac{2}{3}}]+1}\binom{[tn^{\frac{2}{3}}]+1}{\frac{m_{2}+[tn^{\frac{2}{3}}]+2}{2}} \\
& + C\left( \frac{1}{2} \right)^{m+[tn^{\frac{2}{3}}]}\sum_{N}\sum_{l=1}^{N}\sum_{p_{1},\ldots,p_{N}~|~ p_{l}\le m}\sum_{q_{1},\ldots,q_{N-1}}\sum_{r_{1},\ldots,r_{N}} n^{\frac{m+[tn^{\frac{2}{3}}]}{2}-N+1}3^{N}\\
&~~~~~~~~~\left(  m+1\right)\frac{m_{1}}{m+1} \binom{m+1}{\frac{m_{1}+m+1}{2}}\frac{m_{2}+1}{[tn^{\frac{2}{3}}]+1}\binom{[tn^{\frac{2}{3}}]+1}{\frac{m_{2}+[tn^{\frac{2}{3}}]+2}{2}} +o(1)
\end{split}
\end{equation}
As the sum over all the unconstrained case (i.e. when both $l(w_{1})$ and $l(w_{2})$ is of $O(n^{\frac{2}{3}})$) is of $O(1)$, the sum in \eqref{eq:tobound} goes to $0$.\\
\textbf{Step 3 (Analysis of $T_{2}$)}
Now we consider the case when $E_{a_{1}} \cap E_{a_{2}}= \emptyset$. Observe that in this case  the random variables $\prod_{i=1}^{l}\left(X_{w_{i}}-E[X_{w_{i}}]\right)$ and $\prod_{j=1}^{l'}\left( X_{w'_{j}}-\E[X_{w'_{j}}] \right)$ are independent. As a result, 
\begin{equation}
\begin{split}
&\E\left[\prod_{i=1}^{l}\left(X_{w_{i}}-E[X_{w_{i}}]\right)\prod_{j=1}^{l'}\left( X_{w'_{j}}-\E[X_{w'_{j}}] \right)\right]\\
&= \E\left[ \prod_{i=1}^{l}\left(X_{w_{i}}-E[X_{w_{i}}]\right) \right]\E\left[ \prod_{j=1}^{l'}\left( X_{w'_{j}}-\E[X_{w'_{j}}] \right) \right]
\end{split}
\end{equation}
By analysis of $T_{1}$, 
\begin{equation}
\begin{split}
&\E\left[\prod_{i=1}^{l} \left(\Tr\left[ W^{m_{i}'} \right] - \E\left[\Tr\left[ W^{m_{i}'} \right]\right]\right) \prod_{j=1}^{l'}\left(\Tr\left[W^{[t_{j}n^{\frac{2}{3}}]}\right] - \E\left[  \Tr\left[W^{[t_{j}n^{\frac{2}{3}}]}\right]\right] \right)  \right]\\
&= \left( \frac{1}{n} \right)^{\sum_{i}\frac{m_{i}'}{2}+\sum_{j}\frac{[t_{j}n^{\frac{2}{3}}]}{2}}  \sum_{a_{1},a_{2}~|~E_{a_{1}}\cap E_{a_{2}=\emptyset}} \E\left[ \prod_{i=1}^{l}\left(X_{w_{i}}-E[X_{w_{i}}]\right) \right]\E\left[ \prod_{j=1}^{l'}\left( X_{w'_{j}}-\E[X_{w'_{j}}] \right) \right]  +o(1).
\end{split}
\end{equation}
Now we know that 
\begin{equation}
\begin{split}
&\E\left[ \prod_{i=1}^{l} \left(\Tr\left[ W^{m_{i}'} \right] - \E\left[\Tr\left[ W^{m_{i}'} \right]\right]\right) \right]\E\left[\prod_{j=1}^{l'}\left(\Tr\left[W^{[t_{j}n^{\frac{2}{3}}]}\right] - \E\left[  \Tr\left[W^{[t_{j}n^{\frac{2}{3}}]}\right]\right] \right)   \right]\\
&= \left( \frac{1}{n} \right)^{\sum_{i}\frac{m_{i}'}{2}+\sum_{j}\frac{[t_{j}n^{\frac{2}{3}}]}{2}}  \sum_{a_{1},a_{2}} \E\left[ \prod_{i=1}^{l}\left(X_{w_{i}}-E[X_{w_{i}}]\right) \right]\E\left[ \prod_{j=1}^{l'}\left( X_{w'_{j}}-\E[X_{w'_{j}}] \right) \right]\\
&= T_{2} + T_{1}'
\end{split}
\end{equation}
where 
\begin{equation}
T_{1}'= \sum_{a=[a_{1},a_{2}]\in \mathcal{A}_{\ge 2,s}}\E\left[ \prod_{i=1}^{l} \left(\Tr\left[ W^{m_{i}'} \right] - \E\left[\Tr\left[ W^{m_{i}'} \right]\right]\right)\right] \E\left[\prod_{j=1}^{l'}\left(\Tr\left[W^{[t_{j}n^{\frac{2}{3}}]}\right] - \E\left[  \Tr\left[W^{[t_{j}n^{\frac{2}{3}}]}\right]\right] \right) \right].
\end{equation}
Here $\mathcal{A}_{\ge 2, s}$ denotes all sentences $a=(a_{1},a_{2})$ such that each edge in the graph $G_{a}$ is traversed at least twice and $E_{a_{1}} \cap E_{a_{2}}\neq \emptyset.$ It is easy to see (by Holder's inequality) that $|T_{1}'|\le |T_{1}|$. Hence $|T_{1}'| \to 0.$ As a consequence, 
\begin{equation}
\begin{split}
&\left| \E\left[\prod_{i=1}^{l} \left(\Tr\left[ W^{m_{i}'} \right] - \E\left[\Tr\left[ W^{m_{i}'} \right]\right]\right) \prod_{j=1}^{l'}\left(\Tr\left[W^{[t_{j}n^{\frac{2}{3}}]}\right] - \E\left[  \Tr\left[W^{[t_{j}n^{\frac{2}{3}}]}\right]\right] \right)  \right]- \right.\\
&~~~~~\left. \E\left[ \prod_{i=1}^{l} \left(\Tr\left[ W^{m_{i}'} \right] - \E\left[\Tr\left[ W^{m_{i}'} \right]\right]\right) \right]\E\left[\prod_{j=1}^{l'}\left(\Tr\left[W^{[t_{j}n^{\frac{2}{3}}]}\right] - \E\left[  \Tr\left[W^{[t_{j}n^{\frac{2}{3}}]}\right]\right] \right)   \right]  \right| \to 0.
\end{split}
\end{equation}
This completes the proof.
\end{proof}
\begin{proof}[Proof of Theorem \ref{thm:main}]
We only prove part $2$ as proof of part $1$ is contained in the proof of Theorem \ref{thm:joincltpp} and part $3$ is a straight forward corollary of part $2$.

This is enough to prove that for an analytic function $g$ of polynomial growth on a set containing $[-1-\varepsilon, 1+\varepsilon]$ and any bounded continuous function $q: \mathbb{R} \to \mathbb{R}$ we have 
\begin{equation}
\begin{split}
&\lim_{n \to \infty} \E\left[ q\left( \Tr[g(W)]- \E[\Tr[g(W)]] \right) \prod_{j=1}^{l'}\left(\Tr\left[W^{[t_{j}n^{\frac{2}{3}}]}\right] - \E\left[  \Tr\left[W^{[t_{j}n^{\frac{2}{3}}]}\right]\right] \right)\right]\\
&=\lim_{n \to \infty}\E\left[q\left( \Tr[g(W)]- \E[\Tr[g(W)]] \right)\right]\lim_{n \to \infty}\E\left[ \prod_{j=1}^{l'}\left(\Tr\left[W^{[t_{j}n^{\frac{2}{3}}]}\right] - \E\left[  \Tr\left[W^{[t_{j}n^{\frac{2}{3}}]}\right]\right] \right) \right]
\end{split}
\end{equation}

We shall fix any $\eta >0$ , choose $i$ large enough and compare $q\left( \Tr[g(W)]- \E[\Tr[g(W)]] \right)$ and $q\left( \Tr[g^{(i)}(W)]- \E[\Tr[g^{(i)}(W)]] \right)$. Here $g^{(i)}(x)=\sum_{j=0}^{i} g_{j}x^{j}$. Since both \\ $\left( \Tr[g(W)]- \E[\Tr[g(W)]] \right)$ is a Gaussian random variable in the limit and\\ $\E\left( \Tr[g^{(i)}(W)]- \E[\Tr[g^{(i)}(W)]] \right)^2$ is uniformly bounded over $n$ and $i$, they are tight uniformly over $n$ and $i$. In particular we can choose $M$ large enough so that both\\ 
$P\left[ \left| \Tr[g(W)]- \E[\Tr[g(W)]] \right| > M \right]$ and $P\left[ \left| \Tr[g^{(i)}(W)]- \E[\Tr[g^{(i)}(W)]] \right|> M \right]$ are less than $\eta$ uniformly over $n$ and $i$. Now inside $[-M,M]$ $q$ is uniformly continuous. So given any $0< \eta'<\eta $ there exists $\delta>0$ such that $|x-y|<\delta$ implies $|q(x)-q(y)|<\eta'$. Now we choose $i$ large enough so that $\Var\left[ \Tr[g(W)]- \Tr[g^{(i)}(W)] \right]\le \delta^{2}\eta$. As a consequence, by Chebyshev's inequality $$P\left[ \left| q\left( \Tr[g(W)]- \E[\Tr[g(W)]] \right) - q\left( \Tr[g^{(i)}(W)]- \E[\Tr[g^{(i)}(W)]] \right) \right|> \eta' \right]< \eta + 2\eta=3\eta.$$ Hence, 
$$ \E\left[ \left( q\left( \Tr[g(W)]- \E[\Tr[g(W)]] \right) - q\left( \Tr[g^{(i)}(W)]- \E[\Tr[g^{(i)}(W)]] \right)\right)^2 \right] \le 3\eta || q||_{\infty}^2 + \eta^2 $$. Now 
\begin{equation}\label{eq:bignonsense}
\begin{split}
& \limsup_{n}\left| \E\left[ q\left( \Tr[g(W)]- \E[\Tr[g(W)]] \right) \prod_{j=1}^{l'}\left(\Tr\left[W^{[t_{j}n^{\frac{2}{3}}]}\right] - \E\left[  \Tr\left[W^{[t_{j}n^{\frac{2}{3}}]}\right]\right] \right)\right]-  \right.\\
&~~~~~ \left.\E\left[q\left( \Tr[g(W)]- \E[\Tr[g(W)]] \right)\right]\E\left[ \prod_{j=1}^{l'}\left(\Tr\left[W^{[t_{j}n^{\frac{2}{3}}]}\right] - \E\left[  \Tr\left[W^{[t_{j}n^{\frac{2}{3}}]}\right]\right] \right) \right]\right|\\
& =\limsup_{n} \left| \E\left[ q\left( \Tr[g(W)]- \E[\Tr[g(W)]] \right) \prod_{j=1}^{l'}\left(\Tr\left[W^{[t_{j}n^{\frac{2}{3}}]}\right] - \E\left[  \Tr\left[W^{[t_{j}n^{\frac{2}{3}}]}\right]\right] \right)\right]- \right.\\
& ~~~~~\E\left[ q\left( \Tr[g^{(i)}(W)]- \E[\Tr[g^{(i)}(W)]] \right) \prod_{j=1}^{l'}\left(\Tr\left[W^{[t_{j}n^{\frac{2}{3}}]}\right] - \E\left[  \Tr\left[W^{[t_{j}n^{\frac{2}{3}}]}\right]\right] \right)\right] +\\
&~~~~~\E\left[ q\left( \Tr[g^{(i)}(W)]- \E[\Tr[g^{(i)}(W)]] \right) \prod_{j=1}^{l'}\left(\Tr\left[W^{[t_{j}n^{\frac{2}{3}}]}\right] - \E\left[  \Tr\left[W^{[t_{j}n^{\frac{2}{3}}]}\right]\right] \right)\right]-\\
& ~~~~~ \left. \E\left[q\left( \Tr[g(W)]- \E[\Tr[g(W)]] \right)\right]\E\left[ \prod_{j=1}^{l'}\left(\Tr\left[W^{[t_{j}n^{\frac{2}{3}}]}\right] - \E\left[  \Tr\left[W^{[t_{j}n^{\frac{2}{3}}]}\right]\right] \right) \right] \right|\\
& \le \limsup_{n} \left|\E\left[ q\left( \Tr[g(W)]- \E[\Tr[g(W)]] \right)- q\left( \Tr[g^{(i)}(W)]- \E[\Tr[g^{(i)}(W)]] \right) \right]\times \right.\\
&~~~~\left. \prod_{j=1}^{l'}\left(\Tr\left[W^{[t_{j}n^{\frac{2}{3}}]}\right] - \E\left[  \Tr\left[W^{[t_{j}n^{\frac{2}{3}}]}\right]\right] \right) \right|+ \left|\E\left[ \prod_{j=1}^{l'}\left(\Tr\left[W^{[t_{j}n^{\frac{2}{3}}]}\right] - \E\left[  \Tr\left[W^{[t_{j}n^{\frac{2}{3}}]}\right]\right] \right) \right]\times \right. \\
&~~~~~ \left.\E\left[ q\left( \Tr[g(W)]- \E[\Tr[g(W)]] \right)- q\left( \Tr[g^{(i)}(W)]- \E[\Tr[g^{(i)}(W)]] \right) \right]\right| +o(1)\\
& \le 2\sqrt{3\eta || q||_{\infty}^2 + \eta^2}\limsup\left[\E\left(  \prod_{j=1}^{l'}\left(\Tr\left[W^{[t_{j}n^{\frac{2}{3}}]}\right] - \E\left[  \Tr\left[W^{[t_{j}n^{\frac{2}{3}}]}\right]\right] \right)\right)^{2} \right]^{\frac{1}{2}}+o(1)
\end{split}
\end{equation}
Here the last line follows from Cauchy-Schwarz inequality. As $\eta>0$ was arbitrary, this proves that the first expression of \eqref{eq:bignonsense} goes to $0$. This proves the result.
\end{proof}

\noindent
\textbf{Acknowledgment:} This research is supported by an Inspire faculty fellowship. I am highly grateful to Prof. Krishna Maddaly for his interest and encouragement.
 \bibliography{PAR_SPI}
 \end{document}